\documentclass[10pt]{amsart}

\usepackage{amssymb, amscd, amsmath, amsthm, epsf, epsfig, latexsym,  psfrag}
\usepackage{pb-diagram, fancyhdr, graphicx}

\newcommand{\compactlist}{\begin{list}{$\bullet$}{\setlength{\leftmargin}{1em}}}
\def\zz{{\bf Z}}

\def\qq{{\bf Q}}
\def\cc{{\bf C}}

\def\co{\colon\thinspace}

\def\calg{\mathcal{G}}

\newtheorem{theorem}{Theorem}
\newtheorem{lemma}[theorem]{Lemma}
\newtheorem{corollary}[theorem]{Corollary}

\theoremstyle{definition}
\newtheorem{definition}[theorem]{Definition}
\def\co{\colon\thinspace}

\numberwithin{equation}{section}

\begin{document}

\title{Concordance of Bing Doubles and Boundary Genus}
 \author{Charles Livingston}\author{Cornelia Van Cott}
\thanks{This work was partially supported by  NSF-DMS-0707078 and  NSF-DMS-1007196}
 
\address{Charles Livingston: Department of Mathematics, Indiana University, Bloomington, IN 47405 }
\email{livingst@indiana.edu}
\address{Cornelia Van Cott: Department of Mathematics, University of San Francisco, San Francisco, CA  94117}
\email{cvancott@usfca.edu}


 \begin{abstract} Cha and Kim proved that if a knot $K$ is not algebraically slice, then no iterated Bing double of $K$ is concordant to the unlink.  We prove that if $K$ has nontrivial signature $\sigma$, then the $n$--iterated Bing double of $K$ is not concordant to any boundary link with boundary surfaces of genus less than $2^{n-1}\sigma$.  The same result holds with $\sigma$ replaced by $2\tau$, twice the Ozsv\'ath-Szab\'o knot concordance invariant.\end{abstract}

\maketitle

\section{Introduction}
 
The construction of Bing doubling, introduced in~\cite{bing},  converts a knot $K \subset S^3$ into a two component link $B(K)$.  The construction can be iterated, producing an entire family of links $\{B_n(K)\}$, where $B_n(K)$ has $2^n$ components.  Figure~\ref{example} illustrates the figure eight knot $K$, its Bing double $B(K)$, and the second iterated Bing double $B_2(K)$.   

From several perspectives, Bing doubles are indistinguishable from each other.  For example, all Bing doubles have vanishing Arf invariants, vanishing multivariable Tristram-Levine signatures, and trivial multivariable Alexander polynomial~\cite{cim}.  In addition, all Bing doubles are boundary links, and hence have vanishing Milnor $\bar{\mu}$ invariants~\cite{cim}.  Moreover, the links $B_n(K)$ bound identical embedded $2^n$ component surfaces in $B^4$ as follows: for every knot $K$, $B_n(K)$ bounds a surface where one component  is a punctured torus and the rest are disks (see Section~\ref{sectionbackground}).  

The perspective from which Bing doubles become interesting is the perspective of concordance.  Some -- but not all -- Bing doubles are slice.  Work of Casson and Freedman in four-dimensional surgery first highlighted the importance of understanding when Bing doubles are slice~\cite{fq}.  Recent years have seen significant progress in this direction; references include~\cite{cha, ck, clr, cim, chl, harvey, levine, teich, vc}.  These investigations  culminated with the following result  of   Cha-Kim~\cite{ck}: If $K \subset S^3$ is not algebraically slice, then the $n^{th}$--iterated Bing double of $K$, $B_n(K)$, is not slice. 

\begin{figure}[h]
\begin{center}
\includegraphics[scale=.1]{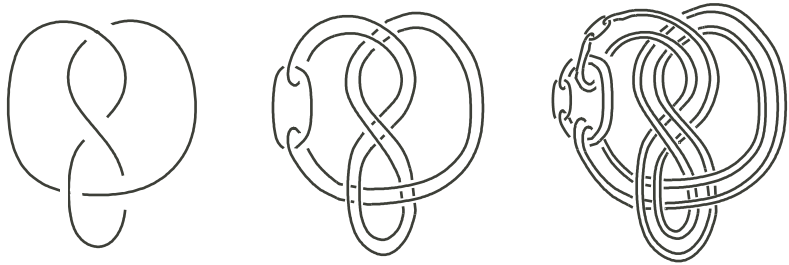}
\end{center}
 \caption{
The figure eight knot,  $K$,   $B(K)$, and   $B_2(K)$. }
\label{example}
\end{figure}

Our results point out additional ways in which iterated Bing doubles are distinguished.  First, in Section~\ref{s3section}, we show that if each member of a family of iterated Bing doubles $\{B_n(K)\}$ bounds disjoint surfaces in $S^3$, the genera of the surfaces increase exponentially in $n$.   (Compare this to the parallel result for surfaces in $B^4$ mentioned above.)  Furthermore, if a link $J$ is concordant to $B_n(K)$, then the genera of the surfaces which the link $J$ bounds will increase exponentially in $n$.  


To make this precise, let $\nu(K)$ be any additive knot invariant for which $\text{genus}(F) \ge |\nu(K)|$ for all $K$ and for all pairs $(W,F)$ where $W$ is a $\zz_{(2)}$--homology ball  and $\partial(W,F) = (S^3, K)$.  We further assume that $\nu(K)=\nu(K^r)$ for all $K$, where $K^r$ denotes the knot $K$ with orientation reversed.   Examples of such invariants include half the Murasugi signature function~\cite{mura}, $\frac{1}{2}\sigma(K),$  the normalized Levine-Tristram $p$--signatures~\cite{le2, tristram},  $\frac{1}{2}\sigma_{i/p}(K)$, and, in the smooth
 category,  the Ozsv\'ath-Szab\'o $\tau$ invariant~\cite{OzsSza}.     (See the appendix for a brief discussion of these invariants why they satisfy the given properties.)  

\begin{theorem} \label{mainthm}   Let $\nu$ be an additive knot invariant with properties   described above.  If the iterated Bing double $B_n(K)$ is concordant to a boundary link $ J = (J_1, \ldots , J_{2^n})  = \partial (F_1, \ldots , F_{2^n}),$ 
then genus($F_i) \ge 2^n|\nu(K)|$ for all $i$. 
\end{theorem}
 
 \vskip.1in
\noindent{\bf Note on coefficients\ }   Notice that one of the properties of $\nu$ include the use of coefficients in $\zz_{(2)}$.  This is the integers localized at two, the ring of rational numbers expressible with odd denominator.  Although much of our work could be modified so that results could be stated in the $\zz$--coefficient setting, the proofs demand working with $\zz_{(2)}$.        Furthermore, the use of $\zz_{(2)}$ (as opposed to $\zz/2\zz$)  is consistent with the past work on the subject which we will be citing.    \vskip.1in

We illustrate an application of this theorem with an example.  Let $  D(T_{2,3}) $ denote the positive Whitehead double of the trefoil knot.  Since $  D(T_{2,3}) $ has trivial Alexander polynomial, it is topologically slice~\cite{freed}.  However, as first shown in~\cite{liv, plam}, the  Ozsv\'ath-Szab\'o $\tau$ invariant   is nontrival.  Thus we have the following corollary.\vskip.1in

\begin{corollary}    Consider the knot $K = D(T_{2,3})$, the untwisted Whitehead double of the trefoil knot.  For all $n$,  $B_n(K)$ is topologically concordant to the unlink,  but if it is smoothly concordant to a boundary link $(J_1, \ldots , J_{2^n})  = \partial (F_1, \ldots , F_{2^n}),$
then genus($F_i) \ge 2^n$ for all $i$.
\end{corollary}
 \vskip.1in
\noindent{\it Acknowledgements} Thanks go to Jae Choon Cha, Stefan Friedl, Matt Hedden, and Danny Ruberman for     conversations  and insights regarding this work. 


\section{The 4--genus of iterated Bing doubles}\label{sectionbackground}
The following result indicates one reason why, from the 4--dimensional perspective, the appropriate measure of complexity for an iterated Bing double $B_n(K)$ is not given by the 4--genus, but rather by the minimum 3--genus of boundary links concordant to $B_n(K)$.

\begin{theorem}For any knot $K \subset S^3$ and $n\ge 1$, there exists a collection of disjoint surfaces $(F_1, \ldots, F_{2^n})$ embedded in $B^4$ such that $\partial (F_1, \ldots, F_{2^n}) = B_n(K)$, genus($F_1) = 1$, and $F_i \cong B^2$ for $i>1$.

\end{theorem}

\begin{proof} It is clear from the diagram that for any one component, say $K_1$, of $B_n(K)$, there is a second component $K_2$ such that by changing two crossings (of opposite sign) between them, they become unlinked.  For instance, in Figure~\ref{example} illustrating $B_2(4_1)$, either of the smallest circles can be unclasped from its adjoining loop by changing two crossings.  Notice that once that one circle is unlinked, the entire link becomes trivial.

The trace of the homotopy corresponding to those crossing changes yields a singular concordance from $B_n(K)$ to the unlink, with the only singularities being a pair of double points between two embedded annuli.  Capping the annuli off with disks in $B^4$ yields a set of singular disks in $B^4$ bounded by $B_n(K)$ with the only singularities being a pair of double points between two of the disks.  

If the two double points are $a$ and $b$ on disks $D_1$ and $D_2$, then disk neighborhood of $a$ and $b$ on $D_1$ can be removed and replaced by an annulus in the boundary of a tubular neighborhood of an embedded arc on $D_2$ that joins $a$ to $b$.  The effect is to convert $D_1$ into a punctured torus that is disjoint from $D_2$ as well as all the other disks.  (That the resulting surface is orientable follows from the fact that the two crossing changes were of opposite sign.) This yields the desired surfaces.

\end{proof}

There are other ways to prove this theorem.  Most simply, any one component bounds a punctured torus in $S^3$ in the complement of the remaining components.  The remaining components form an unlink, so bound disks in $B^4$.  The given proof has the advantage that it also shows that the {\it clasp number} of $B_n(K)$ is 2.  (See~\cite{my} for one discussion of the clasp number of knots.)  The clasp number is another 4--dimensional measure of the complexity of a knot or link, but as we have just seen, it too is of limited value in studying Bing doubles. 

 
 \section{Tree notation for Bing doubling and two covering link theorems}
  Adam Levine~\cite{levine} introduced binary trees into the study of partially iterated Bing doubling, and we will make use of this notation.  Figure~\ref{partial} illustrates a tree diagram corresponding to a partially iterated Bing double of the figure eight knot $K$.  Nodes of the tree correspond to the components of the resulting link.  
  
  \begin{figure}[h]
\begin{center}
\psfrag{k}{\!\!\!$K$}
\includegraphics[scale=.17]{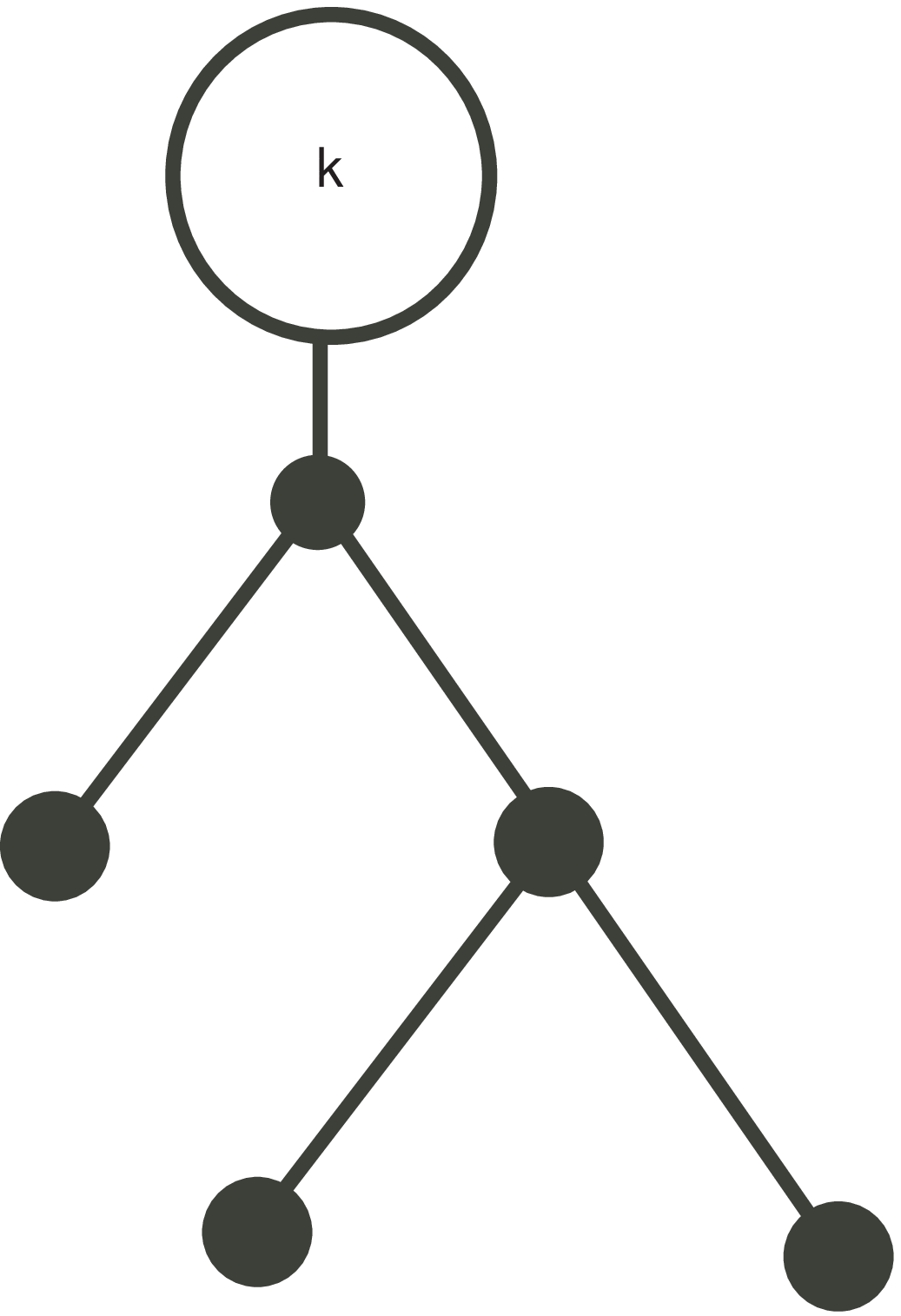}\hskip1in  \includegraphics[scale=.1]{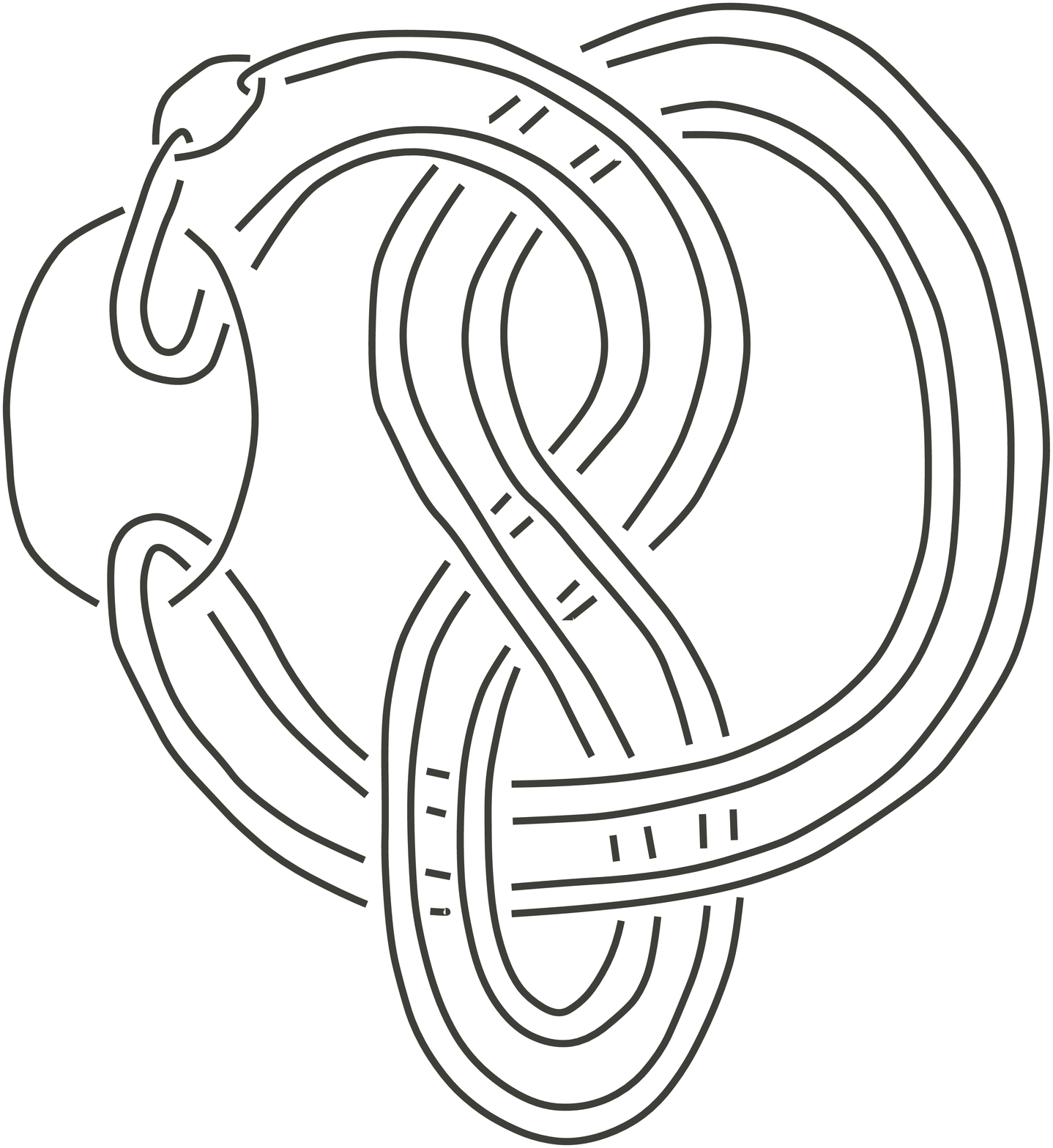}
\end{center}
\caption{}\label{partial}
\end{figure}
  
Let $L$ be a link in $S^3$, and let $M$ denote the 2--fold branched cover of $S^3$, branched over a component of $L$.  Any sublink of the lift of $L$ in $M$ is a {\em covering link} of $L$.  

We give two theorems about covering links.  The first is a generalization of a construction in~\cite{clr}.  In the statement of the theorem, $K^r$ denotes $K$ with its string orientation reversed.

\begin{theorem}[Covering Theorem A] Let $T$ be a binary tree with a marked node of depth 1, and let $T_-$ be the binary tree obtained by deleting this node and its root from $T$, as illustrated in Figure~\ref{formal}.  The link $B_{T_-}(K\#K^r)$ is a covering link of  $B_{T}(K)$.     
\end{theorem}

\begin{figure}[h]
\begin{center}
\psfrag{T}{\Large $T$}
\psfrag{S}{\Large $T_-$}
\includegraphics[width=4cm]{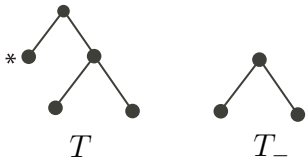}
\end{center}
\caption{A tree $T$ with marked node of depth one, and the associated tree $T_-$.}\label{formal}
\end{figure}

\begin{proof} We begin with the link $B_{T}(K) = (K_0, K_1, \cdots, K_m)$, where $K_0$ is the component of $B_{T}(K)$ corresponding to the marked node of $T$ of depth one.  Because $T$ has a node of depth one, it follows that $B_T(K)$ can be represented by the diagram on the left in Figure~\ref{base}, where the solid torus $S$ contains all components of $B_T(K)$ except for $K_0$.  Taking the 2--fold cover of $S^3$ branched over $K_0$, the remaining components of the link $B_{T}(K)$ lift upstairs to the link   represented on the right in Figure~\ref{base}, where each  of the two  components  of the preimage of $S$ maps homemorphically to $S$. Denote one of those components,  $\widetilde{S}$.  Observe that the sublink of the upstairs link consisting of the components inside the solid torus $\widetilde{S}$ forms $B_{T_-}(K\#K^r)$.  Hence $B_{T_-}(K\#K^r)$ is a covering link of $B_{T}(K)$.  
\end{proof}

\begin{figure}

\psfrag{S}{\Large $S$}
\psfrag{K}{\!\!\Large $K$}
\psfrag{J}{\!\!\!\!\!\Large $K_0$}
\includegraphics[width=4.6cm]{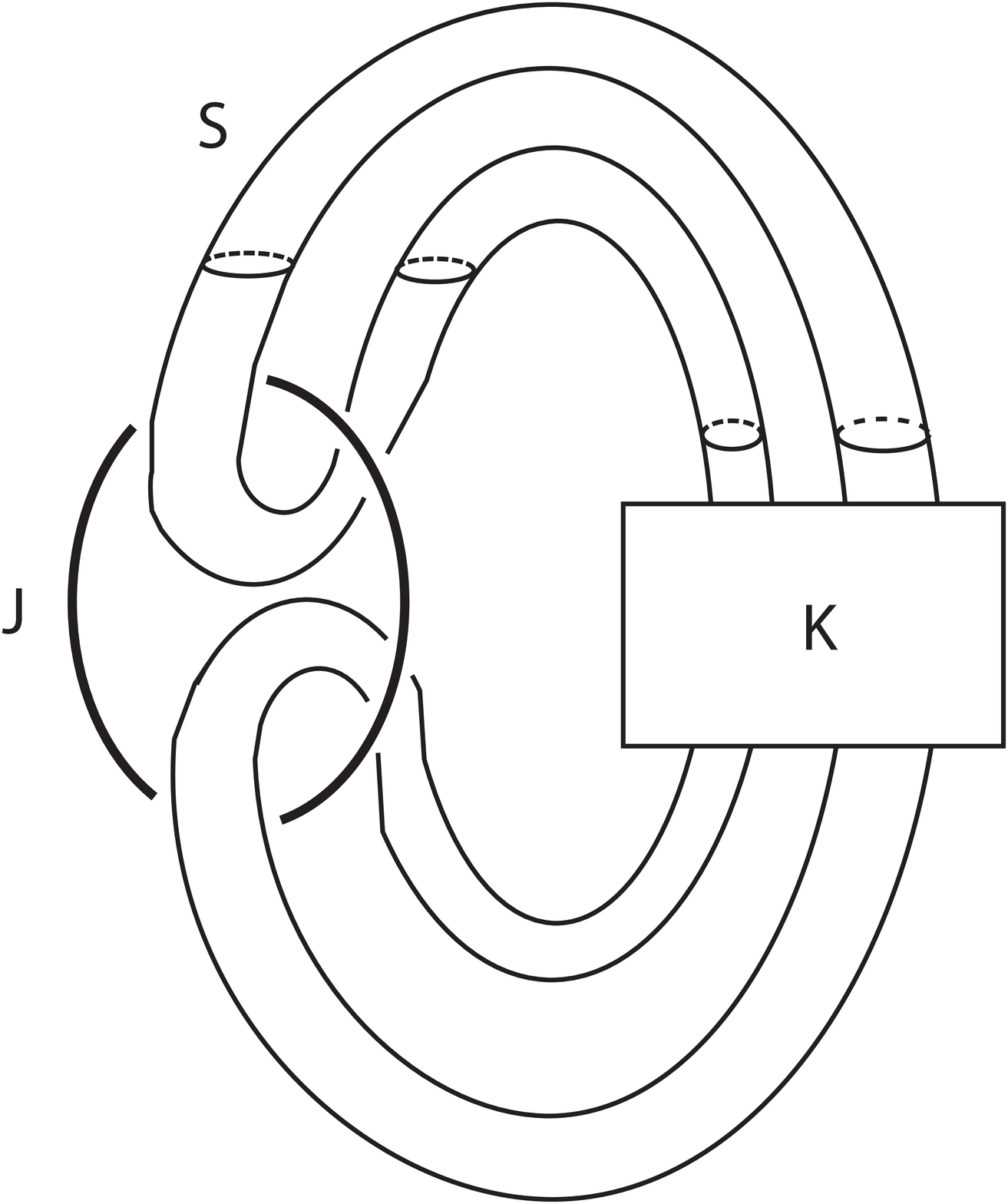}\hskip.4in
\psfrag{S}{\!\!\Large $\widetilde{S}$}
\psfrag{K}{\!\!\Large $K$}
\includegraphics[width=4cm]{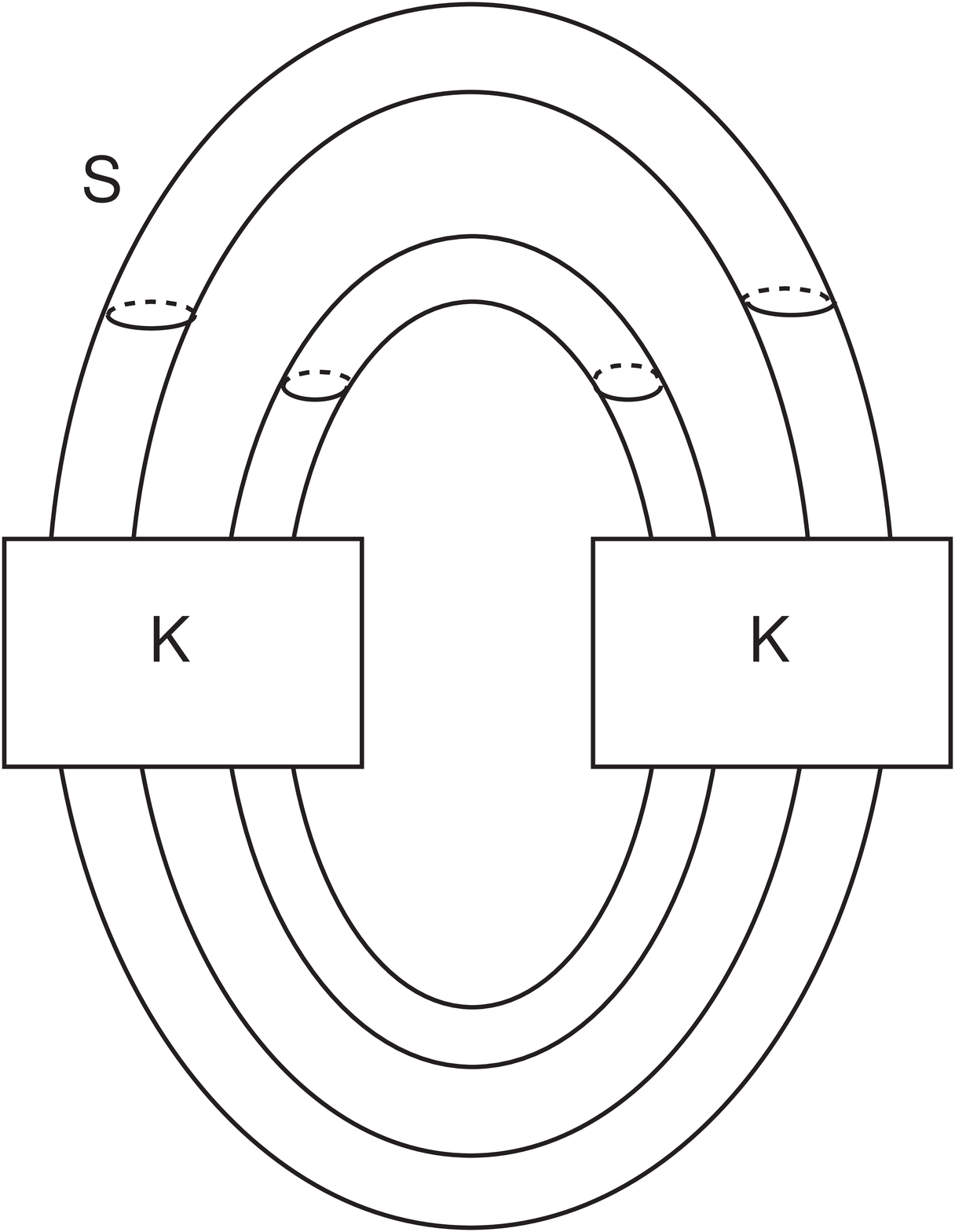}

\caption{$B_T(K)$ and the 2--fold branched cover over $K_0$.}\label{base}
\end{figure}

Our second covering theorem is essentially a result of~\cite{vc}. 

\begin{theorem}[Covering Theorem B]  Let $T$ be a nontrivial binary tree with a marked node, and let $T_+$ be the binary tree obtained by attaching a pair of nodes to the marked node of $T$, as illustrated in Figure~\ref{covB}.  Then $B_{T}(K)$ is a covering link of $B_{T_+}(K)$.

\end{theorem}

\begin{proof}
Since $T$ is nontrivial, the link $B_{T_+}(K) = (K_0^a, K_0^b, K_1, \cdots, K_m)$ can be represented by the diagram on the left in Figure~\ref{link}, where the solid torus $S$ contains all components of the link except for $K_0^a$ and $K_0^b$ and the components $K_0^a$ and $K_0^b$ correspond to the two new nodes on $T_+$.  Taking the 2--fold branched cover of $S^3$ branched over $K_0^a$, the lifts of the remaining components of $B_{T_+}(K)$ in the cover are represented by the diagram on the right in  Figure~\ref{link}.  Observe that the sublink in the right hand diagram in Figure~\ref{link} which is denoted in bold represents the link $B_{T}(K)$.  Hence $B_{T}(K)$ is a covering link of $B_{T_+}(K)$. 
\end{proof}

\begin{figure}[b] 
\begin{center}
\psfrag{T}{\Large $T$}
\psfrag{S}{\Large $T_+$}
\includegraphics[width=5cm]{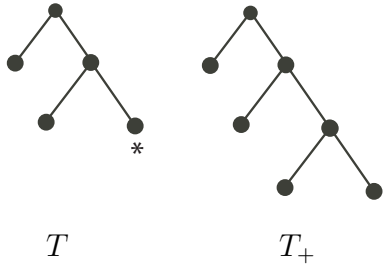}  
\end{center}
\caption{A nontrivial tree $T$ with marked node, and the associated tree $T_+$.}\label{covB}
\end{figure}

\begin{figure}[b]
\begin{center}
\psfrag{J1}{\!\!\!\!\!\!\!\!\! $K_0^a$}
\psfrag{J2}{\!\!\!\!\!\! \!\!\!  \!\!\! $K_0^b$}
\psfrag{T}{\!\!\!   $S$}
\includegraphics[width=6cm]{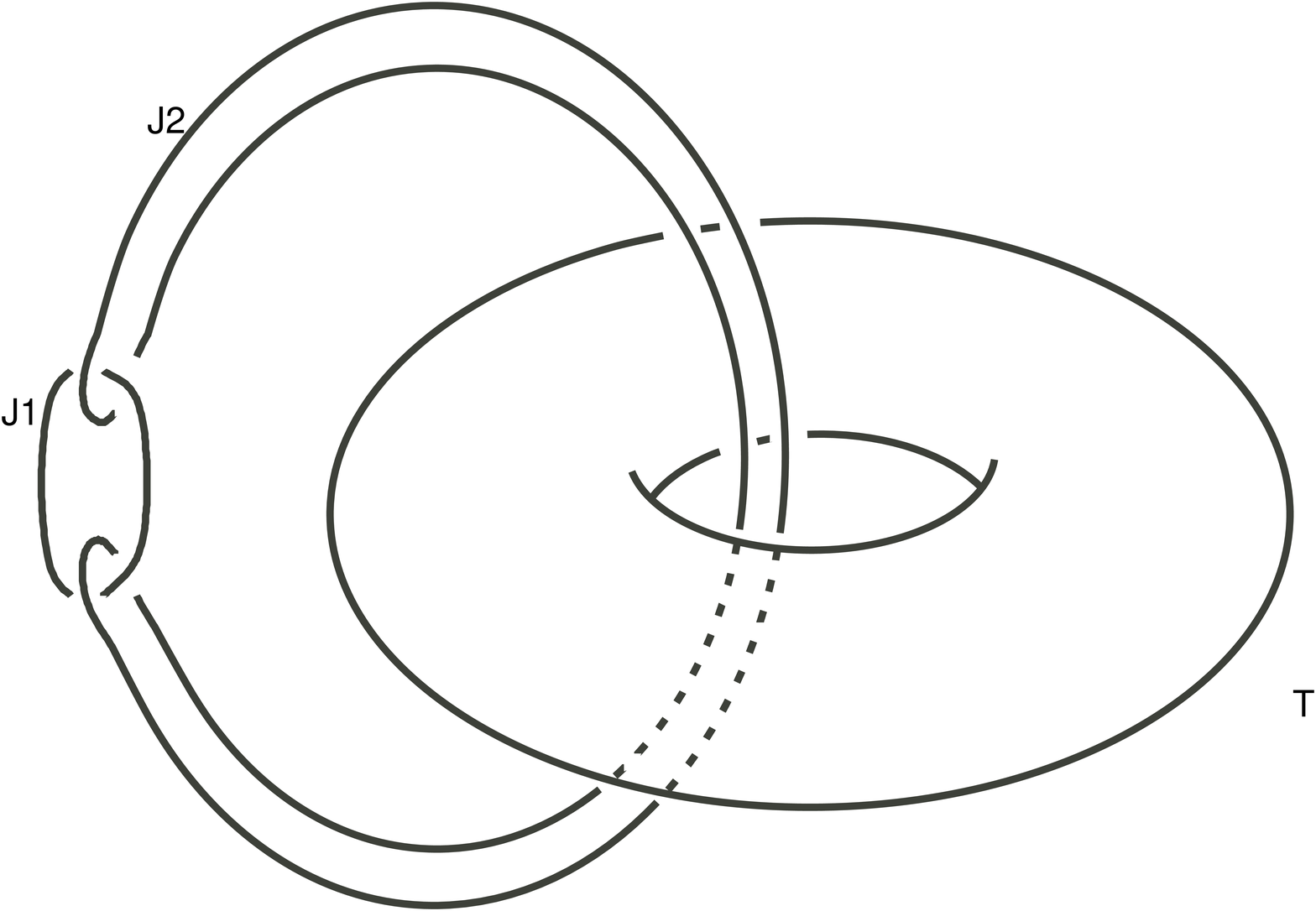}
\psfrag{b}{
}
\psfrag{c}{
}
\psfrag{d}{
}
\psfrag{e}{
}\hskip.2in
\includegraphics[width=6cm]{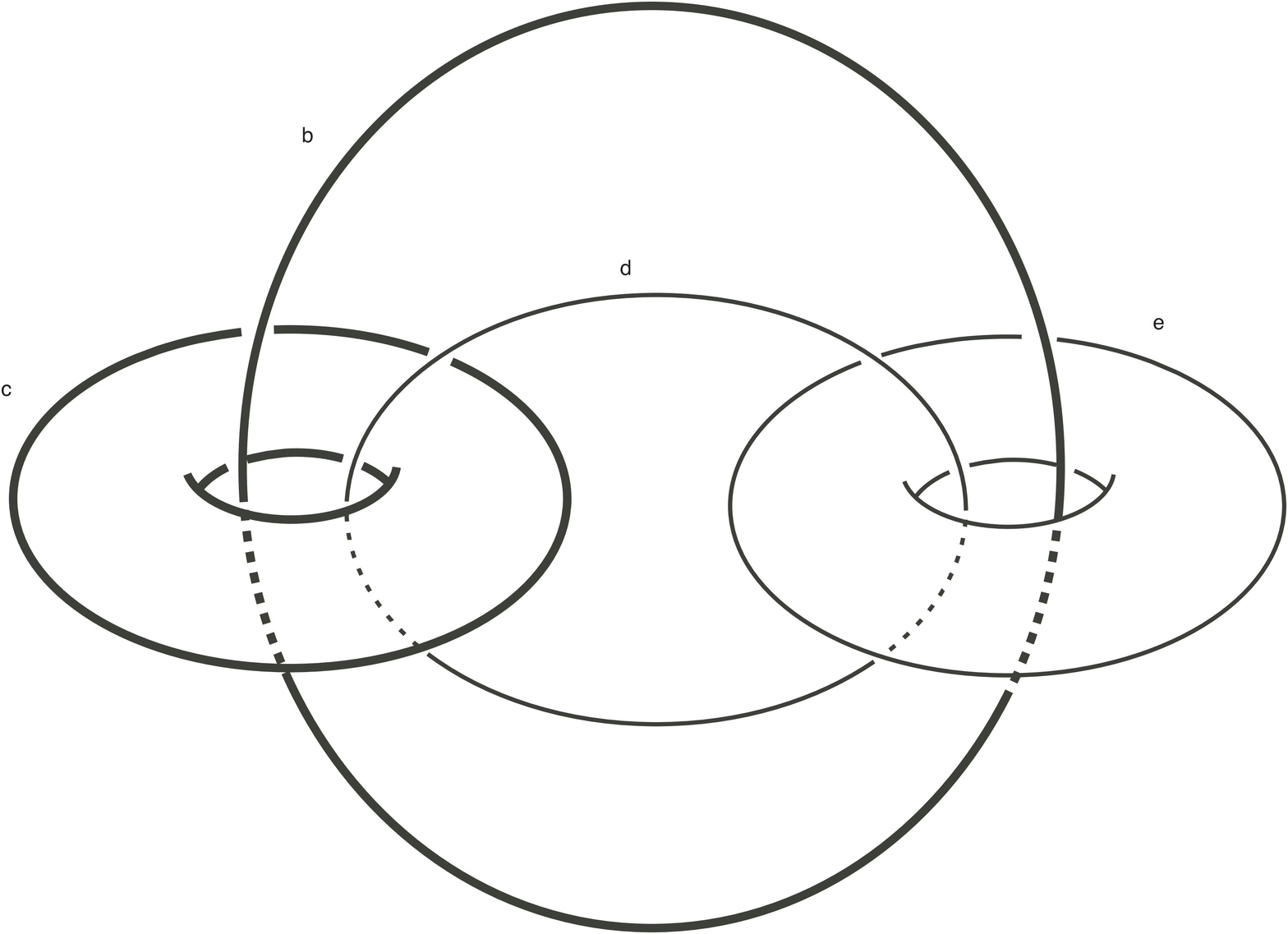}
\end{center}
\caption{$B_{T_+}(K)$ and its branched cover over $K_0^a$.}\label{link}
\end{figure}

Before we move forward in applying these theorems about covering links, we recall the following result concerning the existence of 2--fold branched covers. 

\begin{lemma}\label{coverlemma} Let $M$ be closed 3--manifold satisfying   $H_1(M, \zz_{(2)}) = 0$  and let $K = \partial F$  be a knot  in   $M$.  Then there is a unique connected $2$--fold cover $\widetilde{M}$ of $M$ branched over $K$.  If $G \subset M - F$ is a surface, the preimage of $G$ in  $\widetilde{M}$ consists of two surfaces, each projecting homeomorphically to $G$.  \end{lemma}

\begin{proof}  A Mayer-Vietoris argument implies that $H_1(M-K,\zz_{(2)}) = \zz_{(2)}$.  (See the proof of Lemma~\ref{lem1} for more details.) Since $H_1(M-K, \zz) = \zz^k \oplus T_{odd} \oplus T_{even}$ where  $T_{odd} \oplus T_{even}$ are the odd and even torsion subgroups, by the universal coefficient theorem, $k = 1$ and $T_{even} = 0$. Thus, there is a unique map from $H_1(M-K,\zz)$ onto $\zz/ 2\zz$ and this determines the unique connected 2--fold cover.   

As is true in $S^3$, the branched cover  can be constructed by cutting $M$ open along the bounding surface $F$ for $K$, doubling the resulting space, and gluing the two spaces together along copies of $F$.  Since $F$ misses $G$, it follows that $G$ lifts homeomorphically to the cover.  For details in the $S^3$ case, we refer the reader to~\cite{rolf}.

\end{proof}

\section{The boundary genus of iterated Bing doubles}\label{s3section}

In Section~\ref{sectionbackground}, we saw that every iterated Bing doubles bounds a torus together with a collection of disks in $B^4$.  In this section, we see that if the surfaces are required to live in $S^3$, the situation is not so simple.

In 2006, Cimasoni~\cite{cim} gave a constructive proof of the fact that Bing doubles are boundary links (that is, the components of Bing doubles bound disjoint surfaces in $S^3$).  Given a knot $K$, the surfaces which his proof produces are each of genus $2\cdot g_3(K)$, where $g_3(K)$ denotes the 3--genus of $K$.  Cimasoni's construction naturally extends to iterated Bing doubles to show that for any knot $K$, the components of $B_n(K)$ bound disjoint surfaces in $S^3$: $F_1, \ldots, F_{2^n}$, where genus($F_i) = 2^n g_3(K)$ for all $i$.   In this section, we apply the two covering theorems from the previous section to show that these constructed boundary surfaces are minimal.  

\begin{theorem}\label{cor3} Let $K$ be a knot in $S^3$.   If $B_n(K) = (K_1, \ldots, K_{2^n}) = \partial (F_1, \ldots, F_{2^n})$ where $F_i\subset S^3$, then genus($F_i) \ge 2^n g_3(K)$ for all $i$.
\end{theorem}

\begin{proof} All of the components of $B_n(K)$ are unknotted, so the branched cover of $S^3$ branched over any component of $B_n(K)$ will again be $S^3$.  Moreover, the lifts of $B_n(K)$ in the cover are again unknotted.  Hence, the process of taking covers and lifts can be iterated.  We illustrate the proof in the case of $B_3(K)$, drawn schematically in the diagram on the left in Figure~\ref{b3proof}.  Applying Covering Theorem B four times yields the link illustrated in the second diagram.  Notice that the resulting tree now has a node of depth one.  Applying Covering Theorem A yields the third diagram.  Two more applications of Covering Theorem A yield the fourth and fifth diagrams.  This last diagram represents the knot $4(K \# K^r)\subset S^3$.  

Let $K_1$ denote the component of $B_3(K)$ corresponding to the leftmost node in the tree diagram for the link.  From our previous discussion and Lemma~\ref{coverlemma} concerning bounding surfaces and covers, it follows that the boundary surface $F_1$ for $K_1$ lifts in each cover to bound the component of the link corresponding (again) to the leftmost node in the tree diagram for the covering link.  Continuing to lift from one cover to the next, the surface $F_1$ finally lifts to be a Seifert surface for $ 4(K \# K^r)\subset S^3$.  It follows that genus($F_1) \ge  g_3(4(K \# K^r)) = 8 g_3(K)$.  By modifying the chain of branched covers, one can show that a parallel statement holds for $F_i$, where $i>1$.   
\end{proof}

\begin{figure}[h]
\psfrag{K}{ J}
\begin{center}
\psfrag{A}{\!\!\!\!\!\!\!\!\!\!\!\!$4(K \# K^r)$}
\psfrag{B}{\!\!\!\!\!\!\!\!\!\!\!\!$2(K \# K^r)$}
\psfrag{C}{\!\!\!\!\!\!\!\!\!$ K \# K^r $}
\psfrag{D}{\!\!$ K  $}
\includegraphics[scale=.2]{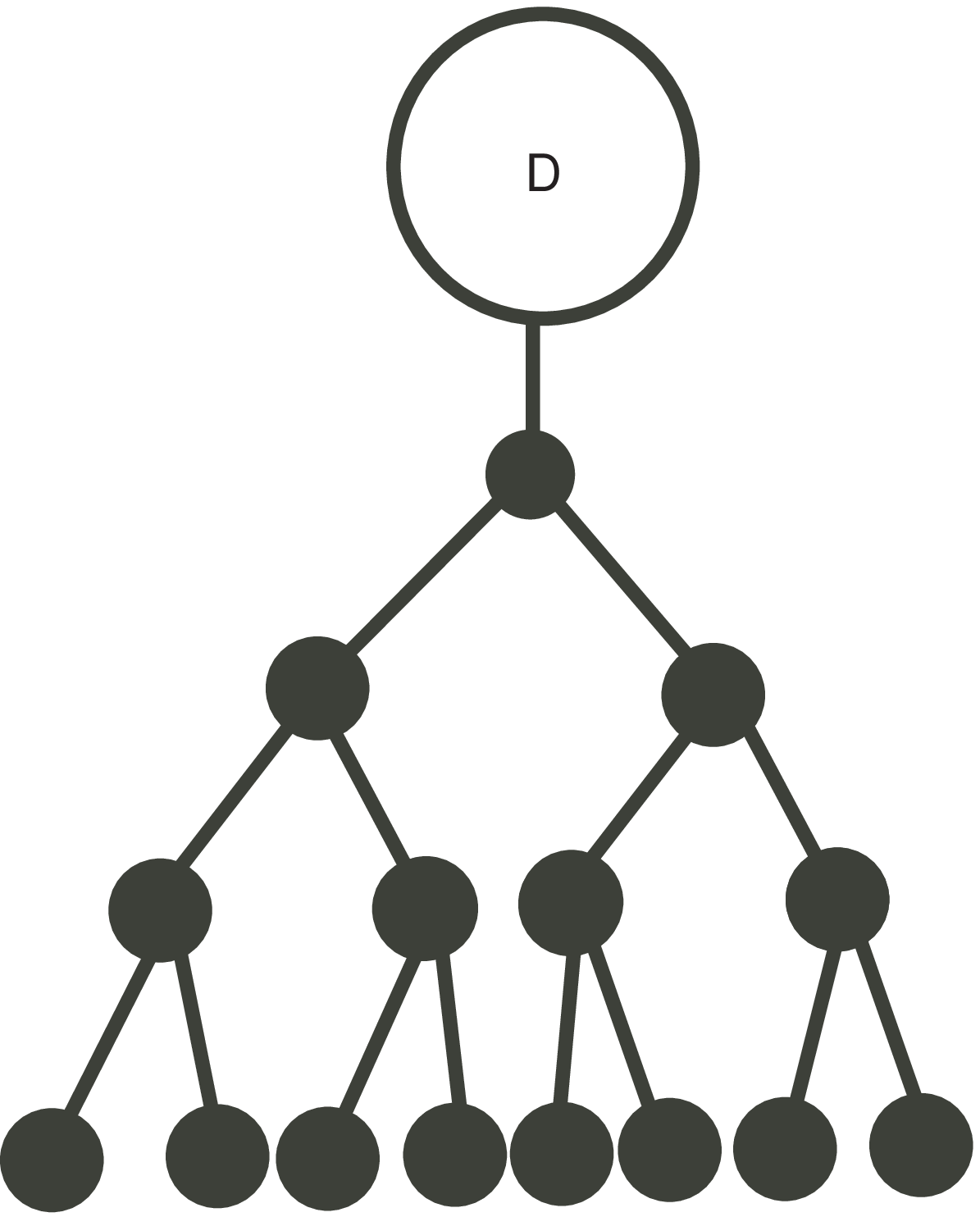} \hskip.1in \includegraphics[scale=.2]{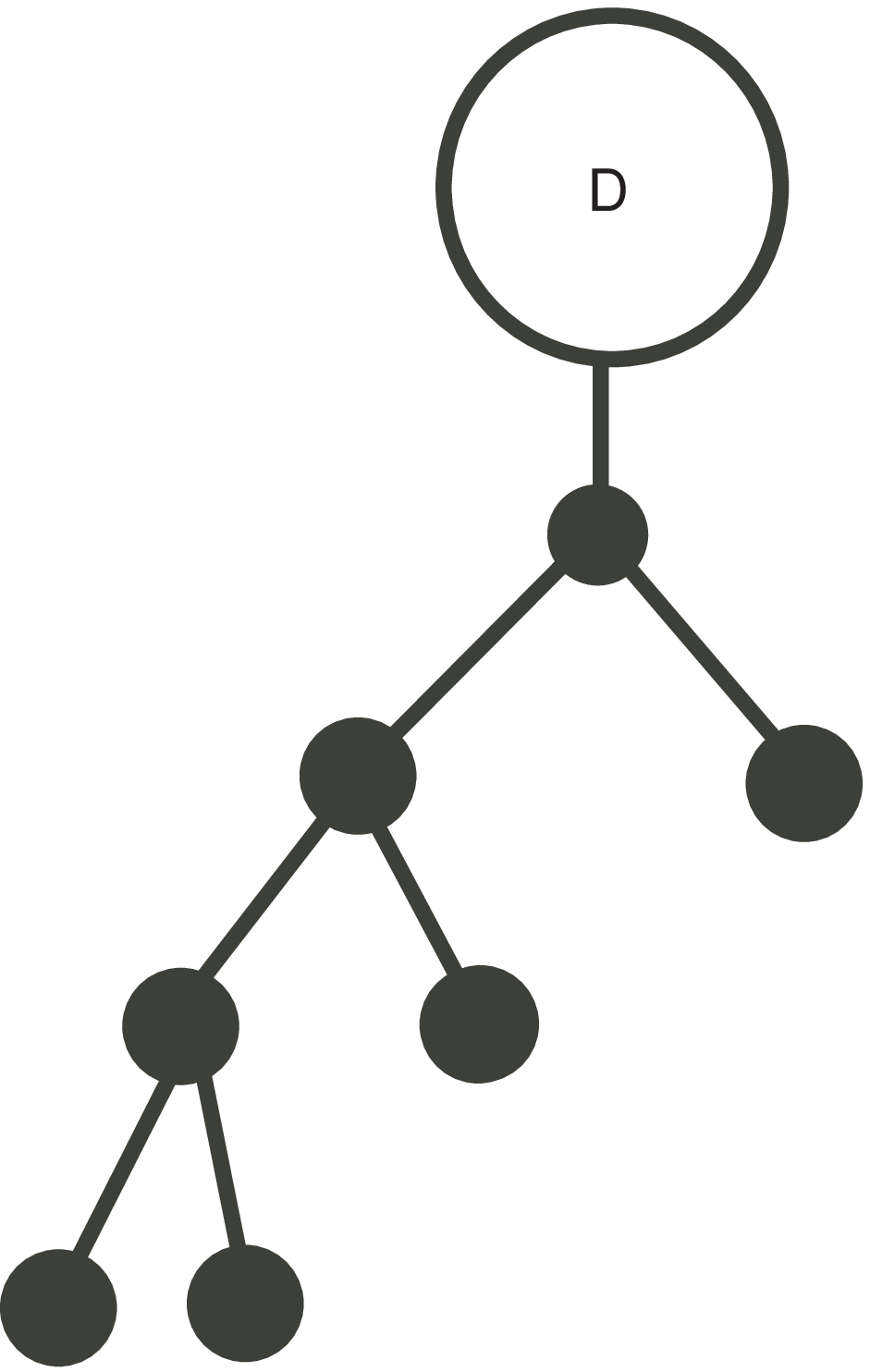}  \hskip.1in \includegraphics[scale=.2]{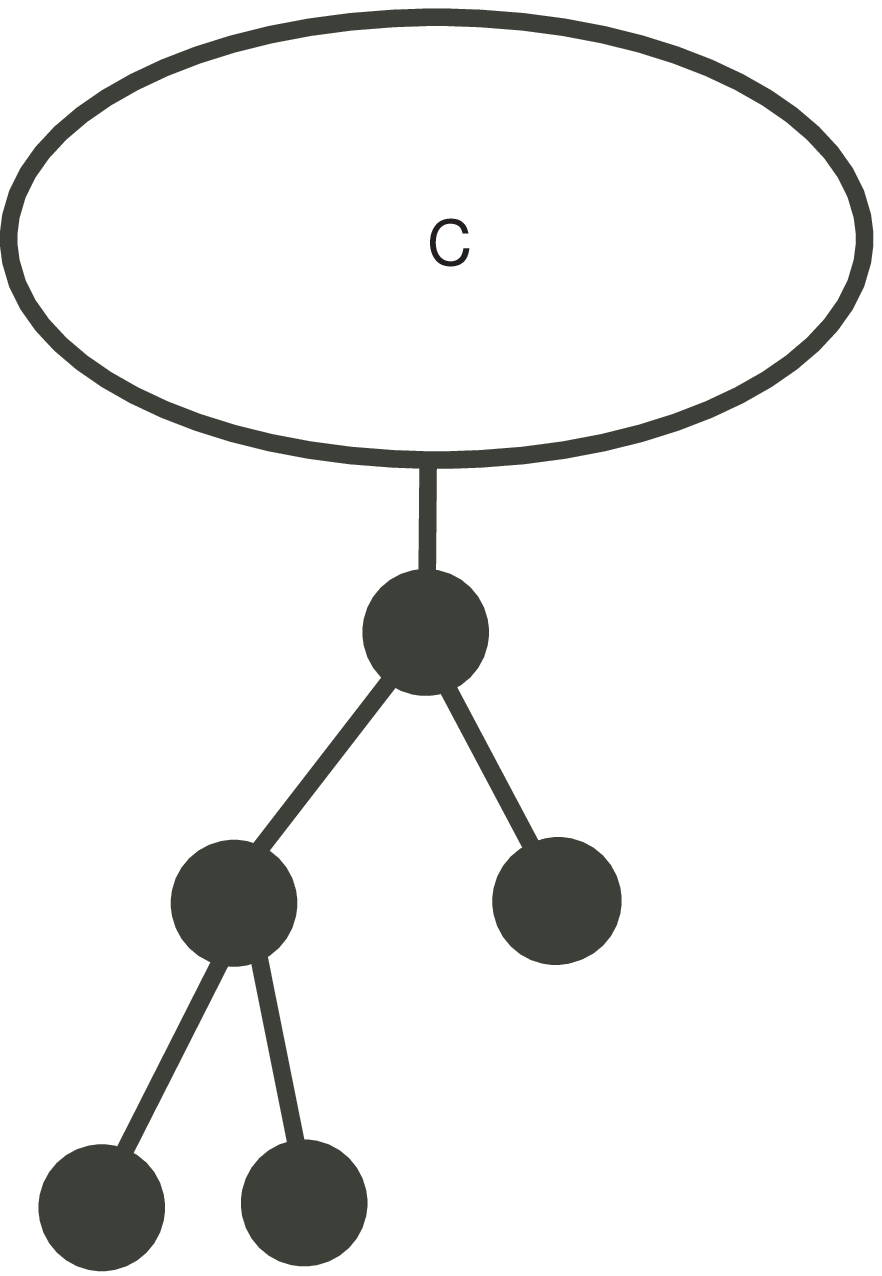}  \hskip.1in \includegraphics[scale=.2]{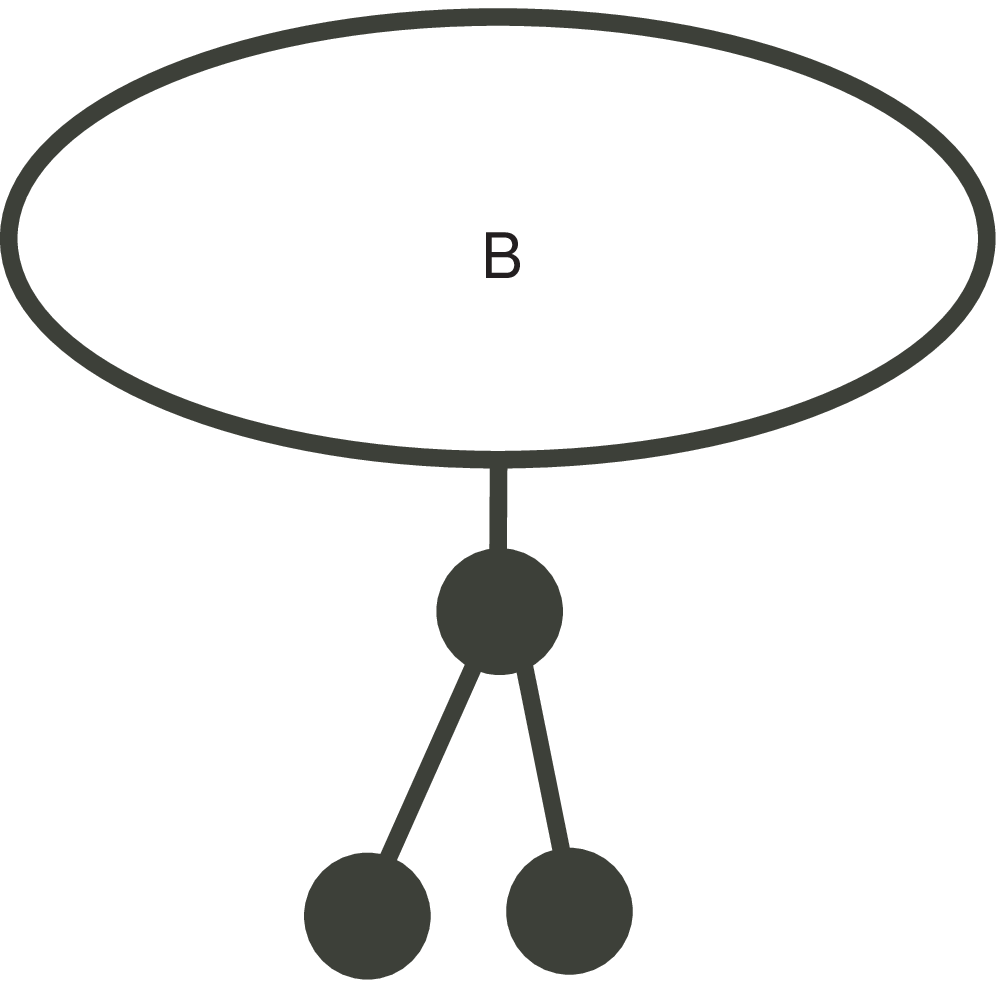}   \hskip.1in \includegraphics[scale=.2]{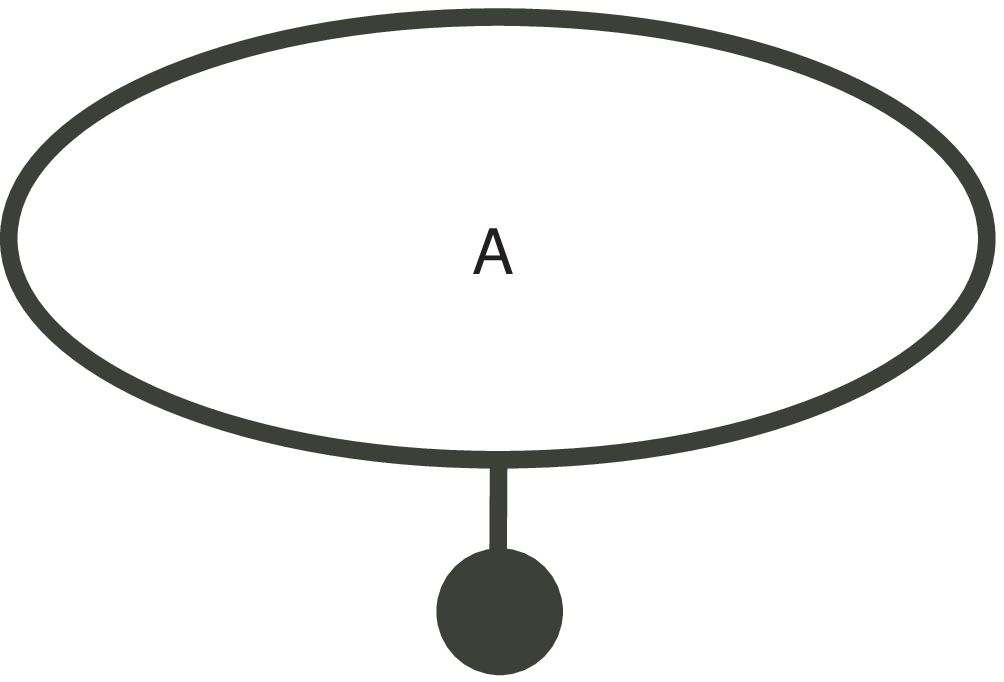}   \end{center}
\caption{A tree representing $B_3(K)$ (far left), and a sequence of trees representing covering links of $B_3(K)$.}\label{b3proof}
\end{figure}

\section{Concordances}~\label{concordances}

We now prove Theorem~\ref{mainthm}, which addresses the boundary genus of links which are concordant to $B_n(K)$.  As stated in the introduction, we assume that $\nu$ is an additive knot invariant for which $\text{genus}(F) \ge |\nu(K)|$ for all $K$ and for all pairs $(W,F)$ where $W$ is a $\zz_{(2)}$--homology ball  and $\partial(W,F) = (S^3, K)$.  We further assume that $\nu(K)=\nu(K^r)$ for all $K$.  Examples of such invariants include half  the Murasugi signature, $\sigma/2$,  the Levine-Tristram $p$--signatures    $\frac{1}{2}\sigma_{i/p}(K)$,  and the 
Ozsv\'ath-Szab\'o $\tau$ invariant.    (See the appendix for details.)
 
The proof of Theorem~\ref{mainthm} will rely on the process of finding a series of branched covers, similar to Theorem~\ref{cor3}.  However, this time the base space is $S^3\times I$, and the branch set will be a component of the concordance from $B_n(K)$ to the link $J$ (so the branch set is a properly embedded $S^1\times I$).  Unlike before, the cover will not necessarily again be homeomorphic to $S^3\times I$.  However, the cover {\em will} have some useful properties.  First, both of the boundary components of the cover will be $\zz_{(2)}$--homology spheres (in particular, one boundary component will be $S^3$, since the branched cover on one end will be branched over a component of $B_n(K)$, which is an unknot).  Moreover, the cover as a whole will have the same homology as a $\zz_{(2)}$--homology sphere.  In addition, the cover will itself admit a 2--fold branched cover, branched over a component of the lift of the concordance.  In Lemma~\ref{lem1}, we verify that the necessary requirements are satisfied for this last statement to hold.

Let $W$   be a compact 4--manifold satisfying $\partial W = M_1 \amalg -M_2$, where $M_1$ is a $\zz_{(2)}$--homology sphere and the inclusion of $M_1$ into $W$ induces an isomorphism on homology.  (With only these assumptions, we will show that $M_2$ is also a $\zz_{(2)}$--homology sphere.)  Let $S$ be a properly embedded $S^1 \times I$   in $W$ with boundary $\partial(W, S) =  (M_1, K_1) \amalg -(M_2, K_2)$.
 
 \begin{lemma}\label{lem1} The inclusion maps  $H_*(M_i - K_i,\zz_{(2)}) \to H_*(W-S, \zz_{(2)})$ are isomorphisms and $H_1(M_1 - K_1, \zz_{(2)}) \cong \zz_{(2)}$.  
 \end{lemma}
 \begin{proof}

Since  $H_*(M_1, \zz_{(2)}) \to H_*(W, \zz_{(2)})$ is an isomorphism, we conclude from the long exact sequence that $H_*(W,  M_1, \zz_{(2)})  = 0$.  Lefschetz duality then implies that $H^*(W,  M_2, \zz_{(2)})  = 0$.  Universal coefficients implies  $H_*(W,  M_2, \zz_{(2)})  = 0$.  Finally, we have from the long exact sequence that $H_*(M_2, \zz_{(2)}) \to H_*(W, \zz_{(2)})$ is an isomorphism.  Hence, $M_2$ is a $\zz_{(2)}$--homology sphere, as well.

Let $X$ be a compact tubular neighborhood of $S$, and let $Y$ be the closure of its complement. Write $ X_1 = X \cap M_1$ and $Y_1= Y\cap M_1$.  Note that $ X \cap Y \cong (S^1 \times S^1) \times I$ and $X_1 \cap Y_1 \cong S^1 \times S^1$.  The inclusion map from $M_1$ into $W$ induces a map from  the Mayer-Vietoris sequence (with $\zz_{(2)}$--coefficients) associated to the decomposition $M_1 = X_1 \cup Y_1$ to the sequence associated to $W = X \cup Y$. The inclusion of $X_1$ into $X$ is a homology equivalence, as is the map from  $X_1 \cap Y_1$ into $X \cap Y$.  The inclusion of $M_1$ into $W$ is assumed to be an isomorphism on homology.  Thus, it follows from the five lemma that the inclusion of $H_*(Y_1, \zz_{(2)})$ into $H_*(Y,\zz_{(2)})$ is an isomorphism.  A similar argument applies for $M_2$.

To prove that $H_1(M_1 - K_1, \zz_{(2)}) = \zz_{(2)}$, we show that $H_1(Y_1,\zz_{(2)}) =  \zz_{(2)}$.  Consider the Mayer-Vietoris sequence $$H_2(M_1 ,  \zz_{(2)}) \to H_1(S^1\times S^1 , \zz_{(2)}) \to H_1(X_1,  \zz_{(2)}) \oplus H_1(Y_1,  \zz_{(2)}) \to H_1(M_1,  \zz_{(2)}).$$  This simplifies to
$$0 \to  \zz_{(2)} \oplus   \zz_{(2)} \to   \zz_{(2)} \oplus H_1(Y_1,  \zz_{(2)}) \to 0.$$ We now see that $H_1(Y_1,\zz_{(2)}) =  \zz_{(2)}$, as desired.

\end{proof}
 
 From Lemma~\ref{lem1}, we can conclude that there is a natural 2--fold branched cover of $W$ branched over $S$, restricting to the natural branched covers of $M_i$ branched over $K_i$.  Denote these covers by $\widetilde{W},  \widetilde{M_1}$, and $\widetilde{M_2}$.  The next result verifies that these covers have the same homological properties as their corresponding spaces downstairs.
 
 \begin{lemma}\label{lem2}
  The cover $\widetilde{M}_1$ is a $\zz_{(2)}$--homology sphere and  the inclusion homomorphism $H_*(\widetilde{M}_i,\zz_{(2)}) \to H_*(\widetilde{W},\zz_{(2)})$ is an isomorphism.

 \end{lemma}
 \begin{proof}  Let $\alpha$ be an embedded arc on $S$ joining a point on $K_1$ to a point on $K_2$.  Let $(Z, Y) = (W - N(\alpha), S-N(\alpha))$, where $N(\alpha)$ is an open tubular neighborhood of $\alpha$ in $W$.  Since the arc is dual to a generator of $H_3(W)$, a homological argument implies that $Z$ is a $\zz_{(2)}$--homology 4--ball and clearly $Y$ is an embedded disk.  According to~\cite{cha}, the 2--fold branched cover of $Z$ branched over $Y$ is a $\zz_{(2)}$--homology 4--ball, which we denote $\widetilde{Z}$.

 The manifold of interest, $\widetilde{W}$, is recovered from $\widetilde{Z}$ by adding a 3--handle.  Since the boundary of $\widetilde{W}$ is disconnected, the 3--handle is added along a separating 2--sphere in $\partial Z$.  Since $H_3(\widetilde{Z}, \zz_{(2)}) = H_2(\widetilde{Z}, \zz_{(2)})$ = 0, we see that $H_3(\widetilde{W}, \zz_{(2)}) = \zz_{(2)}$.  This is the only changed effected by the handle addition, so $ H_4(\widetilde{W}, \zz_{(2)})  = H_2(\widetilde{W}, \zz_{(2)}) = H_1(\widetilde{W}, \zz_{(2)}) = 0$, as desired.

 Note that $\partial \widetilde{Z} = \widetilde{M_1} \# - \widetilde{M_2}$, and thus each $\widetilde{M_i}$ has homology that is isomorphic to that of $\widetilde{W}$.  The fact that the inclusion is an isomorphism will follow from the fact that the inclusion $H_3(\widetilde{M}_1, \zz) \to H_3( \widetilde{W},\zz)$ is an isomorphism.  Dually we show that $H^3(\widetilde{W}, \zz) \to H^3( \widetilde{M}_1,\zz)$ is an isomorphism. This can be achieved by showing that $H^3(\widetilde{W}, \zz) \to H^3( \widetilde{M}_1,\zz) \oplus  H^3( \widetilde{M}_2,\zz)$ is onto the skew diagonal.  But via duality, this is equivalent to showing that the map
 $H_1(\widetilde{W}, \partial  \widetilde{M}_1 \cup \partial  \widetilde{M}_2,  \zz) \to H_0( \widetilde{M}_1,\zz) \oplus  H_0( \widetilde{M}_2,\zz)$ is onto the skew diagonal, a fact that follows quickly from the long exact sequence.
 \end{proof}

 \begin{proof}[\bf Proof of Theorem~\ref{mainthm}] Let $S$ be a concordance in $S^3\times I$ from the link
 $B_n(K) = (K_1, K_2, \ldots, K_{2^n})$ to  the boundary link $J = (J_1, \ldots , J_{2^n}) = \partial (F_1, \ldots , F_{2^n})$.  We show that genus($F_1) \ge 2^n  |\nu  (K)|$.  The argument is similar for $F_i$ where $i>1$.
 
 As seen in the proof of Theorem~\ref{cor3}, there is a sequence of 2--fold branched covers of $S^3$ such that a connected component of the preimage of $K_1$ in the cover is $2^{n-1}(K \# K^r)$. 
 By Lemmas~\ref{lem1} and~\ref{lem2}  there is a corresponding sequence of 2--fold branched covers of $S^3 \times I$ branched over a component of the concordance $S$ between $B_n(K)$ and the link $J$.  We denote the iterated cover of $S^3 \times I$ by $W$.  (By Lemma~\ref{lem2}, each successive cover has the $\zz_{(2)}$--homology of $S^3$, so the process can be iterated.)  At each stage the relevant components of $S$ also lift  to give concordances in the covering spaces.  
 
 Since $J$ is a boundary link, the surfaces $F_i$ lift trivially in each successive cover, and thus $2^{n-1}(K \# K^r)$ is concordant to the lift of $J_1$,  denoted $\widetilde{J}_1$, which bounds the surface $\widetilde{F}_1 \cong F_1$. Denote the lifted surface forming the concordance by $\widetilde{S}_1$.  Notice that $ \widetilde{J}_1 \subset M$ where $M$ is the result of taking the iterated 2--fold covers over the $J_i$ and their lifts.  There is no reason to expect this manifold $M$ to be $S^3$, since the branch sets (a component of $J$ and lifts of components of $J$) may have been nontrivial knots.

 We have now constructed a 4--manifold $W$ which is a $\zz_{(2)}$--homology 3--sphere having boundary $S^3 \amalg -M$ and a concordance in $W$ from   $2^{n-1}(K \# K^r)$ to a knot $\widetilde{J}_1 \subset M$, and $\widetilde{J}_1$ bounds a surface $\widetilde{F}_1$, homeomorphic to $F_1$ in $M$.  Since $W$ provides  $\zz_{(2)}$--homology cobordism from $M$ to  $S^3$, by 
  attaching a $4$--ball to $W$ we see that $M$ bounds a $\zz_{(2)}$--homology ball $B$.  The union of $W$ and $B$ (along $M$) forms a $\zz_{(2)}$--homology ball with boundary $S^3$.  The concordance $\widetilde{S}_1 \cup \widetilde{F}_1$ is a surface  bounded by  $2^{n-1}(K \# K^r) \subset S^3$.    Thus, since $\nu$ provides a lower bound for the genus of such surfaces, we conclude that $\text{genus}(F_1) \ge |\nu(2^{n-1} (K \# K^r))| = 2^n|\nu(K)|$.

 \end{proof}
 
 \section{Conclusion}

There is a homomorphism (described in~\cite{le2}) from the concordance group of knots in $S^3$ onto Levine's algebraic concordance group $\calg$.  It is known that $\calg \cong (\zz/2\zz)^\infty \oplus (\zz/4\zz)^\infty \oplus \zz ^\infty$.  Furthermore, a knot represents torsion in $\calg$ if and only if all the Levine-Tristram signatures vanish.  Thus, one immediate  implication of our results is the following:

 \begin{theorem} If $K$ is of infinite order in Levine's algebraic concordance group $\calg$  and $B_n(K)$ is concordant to a boundary link   $J = \partial (F_1, \ldots, F_{2^n})$, then genus($F_i) \ge 2^n$ for all $i$.
 
 \end{theorem}
 
 This leaves open the case of knots that represent torsion (either geometric or algebraic) in concordance.  The first example is the figure eight knot.  We close with a question:\vskip.2in
 
 \noindent {\bf Problem.}  Let $K$ be the figure eight knot $4_1$.  Suppose that $B_n(K)$ is concordant to a link $J$ and $J = \partial (F_1, \ldots ,F_{2^n})$.  What is the minimum value of genus($F_i)$? 
 
  \appendix
  \section{Signatures as $\zz_{(2)}$ genus bounds.}
  
  In applying Theorem~\ref{mainthm} in the case that $\nu$ is the Levine-Tristram signature, we use the following result.   We need only a weaker result, using $\zz_{(2)}$ coefficients, but the proof is the same in the more general setting.
  
  \begin{theorem}\label{signthm1} Let $(W, F)$ be a pair in which $W$ is compact 4--manifold with $\partial W = S^3$, $H_i(W, \qq) = 0$ for $i>0$, and $F$ has nonempty connected boundary. Then genus($F) \ge \frac{1}{2}\sigma_{i/p}(\partial F)$.
  \end{theorem}
  
Although this is known to experts, we can find no proof in the literature and so outline one here.  The argument depends on a 4--dimensional interpretation of the signature, which we build and show is well-defined in a series of steps.  The structure of the argument is much the same as that presented by Litherland~\cite{litherland}, which applies in the case that $W$ is an integral (as opposed to rational) homology ball.  Related arguments can be found in~\cite{cg}.

\begin{definition} Let $(W,\rho)$ be a pair consisting of a compact 4--manifold $W$ and a homomorphism $\rho \co H_1(W) \to \zz_p$. Define $\sigma(W,  {\rho}) = \sigma_{\omega_p}(\widetilde{W}) - \sigma(W)$, where $\omega_p = e^{2 \pi i /p}$, $\widetilde{W}$ is the associated cyclic $p$--fold cover of $W$, $\sigma(W)$ is the signature of the intersection form on $H_2(W)$, and $ \sigma_{\omega_p}(\widetilde{W})$ is the signature of the intersection form of $H_2(\widetilde{W},\cc)$ restricted to the $\omega_p$--eigenspace of the generator of the  deck transformation of the cover. 
\end{definition}

\begin{theorem}\label{signvanish}  If $W$ is closed, then $\sigma_{\rho}(W) = 0$.
\end{theorem}
\begin{proof}Bordism theory~\cite{cf} implies that the bordism group of pairs $(W,\rho)$ (with $W$ closed) is isomorphic to $\zz$, generated by $CP^2$ with trivial representation.  Thus, forming the connected sum of $(W,\rho)$ with copies of $\pm CP^2$ yields a pair that bounds a compact 5--manifold over $\zz_p$.  For bounding manifolds the signatures vanish, but the connected sum does not alter the difference of signatures.
\end{proof}

\begin{definition}Let $(M,\rho)$ be a pair consisting of a closed 3--manifold with a homomorphism $\rho \co H_1(M) \to \zz_p$.  Bordism theory implies that $p(M,\rho)$ bounds a pair $(W, \tilde{\rho}_p)$.  Define $\sigma(M,\rho) = \frac{1}{p} \sigma(W,  {\rho}) \in \qq$.
\end{definition}

\begin{theorem}  The value of 
$\sigma(M,\rho) $ is independent of the choice of  $(W, \tilde{\rho}_p)$.
\end{theorem}
\begin{proof} Given two choices, $(W_1, \tilde{\rho}^1 )$ and $(W_2, \tilde{\rho}^2 )$, the union $(W_1, \tilde{\rho}^1 ) \cup -(W_2, \tilde{\rho}^2 )$  can be formed.  Now use the additivity of signature under boundary union and Theorem~\ref{signvanish}.
\end{proof}
  
  \vskip.1in
 \noindent{\bf Note} If $\rho$ factors through $\zz$, then, since the bordism group $\Omega_3(\zz)$ vanishes, the pair $(W, \tilde{\rho})$ used to compute $\sigma(M,\rho)$ can be taken to be $p$ disjoint copies of a pair bounded by $(M,\rho)$.    It follows that in this case the signature is an integer. 
 \vskip.1in
 
 \begin{definition} Let $K\subset S^3$ be a knot.  Let $(M(K), \rho)$ be 0--surgery on $K$ with the canonical homomorphism onto $\zz_p$.  We define $\sigma_p(K) = \sigma (M(K),{\rho})$.
 \end{definition}
 
 \begin{theorem}$\sigma_p(K) = \sigma_{1/p}(K)$, the Levine-Tristam signature, defined in terms of a Seifert matrix.
 \end{theorem}
 \begin{proof} For the proof  we could refer directly to Litherland's paper~\cite{litherland}, but we will provide some steps of a construction, since they are needed further on.  
 
 Let $K$ bound a Seifert surface $F$.  Then adding a 0--framed  2--handle to $B^4$ along $K$  yields a 4--manifold $W_1$ with boundary $M(K)$. The generator of $H_2(W_1) = \zz$ is represented by a pushed-in Seifert surface union the core disk of the 2--handle.  This surface we denote $\bar{F}$.
 
 If $F$ has genus $g$, then surgery on $g$ circles in $W_1$ (each lying on $\bar{F}$) yields a 4--manifold $W_2$ bounded by $M(K)$ with $H_2(W_2) = \zz^{2g}$ and $H_1(W_2) = 0$.  The surface $\bar{F}$ can be  simultaneously surgered to be a 2--sphere, $\bar{F}'$.
 
 Now $W_2$ can be surgered along $\bar{F}'$ to yield a 4--manifold $W_3$ with boundary $M(K)$ satisfying $H_2(W_3) = \zz^{2g}$ and $H_1(W_3) = \zz$.  The homomorphism $\rho\co H_1(M(K)) \to \zz_p$ extends to $W_3$, so $W_3$ can be used to compute $\sigma(M(K),\rho)$.  The relevant signatures can be computed explicitly in terms of the Seifert matrix; it follows that $\sigma(M(K),\rho)$ is the signature of the Levine-Tristram matrix $(1 - \omega_p)V + (1 - \bar{\omega}_p)V^t$, as desired.  For details of this final step, see~\cite{litherland}.
 \end{proof}

\begin{theorem}\label{zpthm}If $K\subset S^3$ bounds a surface $F$ of genus $g$ in a compact 4--manifold $W$ where $H_i(W, \qq) = 0$ for $i> 0$, then $|\sigma_p(K)| \le 2g$ for an infinite set of primes $p$.
\end{theorem}

\begin{proof} The construction of the previous proof can be repeated, adding the 2--handle to $W$ instead of to $B^4$.  The result is a 4--manifold $W_3$ with boundary $M(K)$,  $H_3(W_3, \qq) = 0$,   $H_2(W_3, \qq) = \qq^{2g}$, and $H_1(W_3, \qq) = \qq$. Furthermore, the map $  H_1(M(K), \qq) \to   H_1(W_3, \qq) $ is an isomorphism.

From this we conclude  that  $H_1(W_3, \zz) = \zz \oplus G$, where $G$ is a finite group.   It follows that for an infinite set of primes $p$, $H_1(W_3, \zz_p) = \zz_p$. (This holds for any prime that does not divide the order of $G$.    If we further restrict the set of primes so that $p$ does not divide the index of the image of $H_1(M(K))$ in $H_1(W_3)/T \cong \zz$, then we can assume that $H_1(M(K), \zz_p) \to H_1(W_3, \zz_p)$ is an isomorphism.)  Henceforth we select $p$ from this set.

Let $X$ be either $M(K)$ or $W_3$.  There is a unique infinite cyclic cover $\widetilde{X}_\infty \to X$ with deck transformation $T$.  Consider the associated Milnor exact sequence of simplicial chain complexes $$0 \to  C_*(\widetilde{X}_\infty) \xrightarrow{1-T_*}   C_*(\widetilde{X}_\infty) \to C_*( {X}) \to 0.$$
With $\zz_p$ coefficients, we have $$H_1(\tilde{X}_\infty,\zz_p)  \xrightarrow{1-T_*}   H_1(\widetilde{X}_\infty,\zz_p) \to H_1( {X},\zz_p)  \xrightarrow{\phi}  H_0(\widetilde{X}_\infty,\zz_p)  \xrightarrow{1-T_*} H_0 (\widetilde{X}_\infty,\zz_p).$$
The last map is trivial, so the map $\phi$ is an isomorphism (a surjection from $\zz_p$ to itself) and on $H_1(\widetilde{X}_\infty,\zz_p)$ the map $1 - T_*$ is surjective.

It now follows, using $\zz_p$ coefficients, that $1 - (T_*)^p$ is surjective on $ H_1(\widetilde{X}_\infty,\zz_p)$.   Thus, letting $\widetilde{X}$ be the quotient of $\widetilde{X}_\infty$ under the action of $T^p$ (so $\widetilde{X}$ is the $p$--fold cyclic cover of $X$) we have:
 $$H_1(\tilde{X},\zz_p)  \xrightarrow{1- T^p_*}   H_1(\widetilde{X}_\infty,\zz_p) \to H_1( \widetilde{X},\zz_p)  \xrightarrow{\phi}  H_0(\widetilde{X}_\infty,\zz_p)  \xrightarrow{1-T^p_*} H_0 (\widetilde{X}_\infty,\zz_p).$$  The last map is trivial, so, as a consequence, $H_1(\widetilde{X}, \zz_p) = \zz_p$.  

Since $H_1(\widetilde{X})$ maps onto a finite index subgroup of $H_1(X)$, which is infinite, we see that $H_1(\widetilde{X}) = \zz \oplus G$, where $G$ is a finite group of order not divisible by $p$.

The rest of the argument is fairly straightforward.  Using duality and the long exact sequence we find that $H_3(\widetilde{W}, \qq) = 0$.  For the  Euler characteristic of $W$ we have $\chi(W) = 2g$.  By multiplicativity, $\chi(\widetilde{W}) = 2g$.
 Since $\beta_0( \widetilde{W}) = \beta_0( \widetilde{W}) = 1$ 
   and $\beta_3( \widetilde{W}) = 0$, we have $\beta_2 ( \widetilde{W}) = 2pg$.  
  The 1--eigenspace of the deck transformation acting on $H_2 ( \widetilde{W}, \cc) $ is isomorphic
  to $H_2(W, \cc) = \cc^{2g}$.  Thus, each of the $p-1$ other eigenspaces must be exactly $2g$ dimensional, and so  the signature on each 
    is at most $2g$.  The signature of $W$ is 0 (it is built from by surgery from a space with 0 signature, and surgery does not change the signature.)  

This completes the proof.
\end{proof}

\noindent{\bf Comment} Cha notes that a shorter proof of Theorem~\ref{zpthm}, using  more sophisticated tools, is possible. In  his argument, the use of the Milnor exact sequence is replaced with an  application of~\cite[Lemma 3.3]{cha} (which is a consequence of a theorem of Levine~\cite{le3}). Choose any map $W_3 \to S^1$ that is an isomorphism on first homology with $\zz_p$ coefficients.  This can be seen to be 2--connected, so  by~\cite[Lemma 3.3]{cha} we have $H_1(\widetilde{W}_3, \zz_p) \cong H_1(\widetilde{S}^1, \zz_p) \cong \zz_p$.\vskip.2in

The proof of Theorem~\ref{signthm1} now is immediate.  For an infinite set of primes, $\sigma_{1/p}(K)$ has been shown to be bounded by $2g$.  An identical argument, changing eigenvalues, shows that $\sigma_{i/p}(K)$ is bounded by $2g$.  The signature function of $K$, $\sigma_\theta(K)$, is defined for any $\theta$ on the unit circle and is an integer valued step function with a finite number of jump discontinuities, so if it is bounded by $2g$ on a dense set of points, it is bounded by $2g$ everywhere.


\vskip.2in

\newcommand{\etalchar}[1]{$^{#1}$}

\end{document}